\documentclass[11pt]{amsart}
\usepackage{setspace}
\usepackage{amssymb,amsmath,amscd,graphicx,fontenc,amsthm,mathrsfs,hyperref}
\usepackage%[backref=page]
{hyperref}
\hypersetup{colorlinks=true,citecolor=blue,linkcolor=blue,urlcolor=blue,
pdfstartview=FitH, pdfauthor=Valentin Blomer and P\'eter Maga}
\newtheorem{theorem}{Theorem}
\newtheorem{definition}{Definition}
\newtheorem{lemma}{Lemma}
\newtheorem{proposition}{Proposition}
\newtheorem{corollary}{Corollary}
\theoremstyle{definition}
\newtheorem{remark}{Remark}

\usepackage[top=3.8cm,bottom=3.8cm,right=2.9cm,left=2.9cm,twoside=false]{geometry}
\newcommand{\G}{{\mathbf G}}
\newcommand{\GSL}{{\mathbf{SL}}}
\newcommand{\GSp}{{\mathbf{Sp}}}

\newcommand{\ZZ}{{\mathbb{Z}}}
\newcommand{\CC}{{\mathbb{C}}}

\newcommand{\TT}{{\mathbb{T}}}

\newcommand{\FF}{{\mathbb{F}}}

\newcommand{\GL}{\mathrm{GL}}
\newcommand{\SL}{\mathrm{SL}}
\newcommand{\Sp}{{\mathrm{Sp}}}
\newcommand{\Z}[1]{\mathbb{Z}/(#1\mathbb{Z})}
\newcommand{\Alt}{\mathrm{Alt}}

\newcommand{\p}{{\mathrm{pf}}}

\newcommand{\inn}[2]{\langle #1, #2 \rangle}
\newcommand{\pra}[1]{\left( #1 \right)}

\newcommand{\of}{\subseteq}
\newcommand{\sgn}{\mathrm{sgn}}

\newcommand{\pfunction}[4]{
\begin{cases} #1 & \textrm{if } \, #2 \\
#3 & \textrm{ } \, #4
\end{cases}}
\newcommand{\ppfunction}[6]{
\begin{cases} #1 & \textrm{if } \, #2 \\
#3 & \textrm{if } \, #4 \\ #5 & \textrm{if } #6 
\end{cases}}

\begin{document}
\title{Quasi-Random profinite groups}
\author{Mohammad Bardestani}
\address{Mohammad Bardestani, D\'{e}partment de Math\'{e}matiques et Statistique,Universit\'{e} de Montr\'{e}al, CP 6128, succ.
Centre-ville, Montr\'{e}al, QC, Canada H3C 3J7  \newline {\textrm current address:}
Department of Mathematics and Statistics, University of Ottawa, 585 King Edward, Ottawa, ON K1N 6N5, Canada.}
\email{ mbardest@uottawa.ca   }
\author{Keivan Mallahi-Karai}
\address{Keivan Mallahi-Karai, Jacobs University Bremen, Campus Ring I, 28759 Bremen, Germany.}
\email{k.mallahikarai@jacobs-university.de }
%=============================================================
\begin{abstract}
Inspired by Gowers' seminal paper~\cite{Gowers}, we will investigate quasi-randomness for profinite groups. We will
obtain bounds for the mininal degree of non-trivial representations of $\SL_k( \Z{p^n})$ and 
$\Sp_{2k}(\Z{p^n})$. Our method also delivers a lower bound for the minimal degree of a faithful representation 
of these groups. Using the suitable machinery from functional analysis, we establish exponential lower and upper
bounds for the supremal measure of a product-free measurable subset of the profinite groups $\SL_{k}({\ZZ_p})$ and 
$\Sp_{2k}(\ZZ_p)$. We also obtain analogous bounds for a special subgroup of the automorphism group of a regular tree.
%\keywords{Profinite group, Complex representation, Hilbert-Schmidt operator, Singular value decomposition, }
\end{abstract}
\maketitle
%\tableofcontents
%========================================================================================================
%\doublespacing

\section{Introduction}

A subset $A$ of a group $G$ is called product-free if the equation $xy = z$ has no solution with $x,y,z \in A$. Babai and S\'{o}s~\cite{Sos} asked if every finite group $G$ has a product-free subset of size at least $c|G|$ for a universal constant $c>0$. This question was answered negatively by Gowers in his paper on quasi-random groups~\cite{Gowers} where he proved that for sufficiently large prime $p$, the group $G=\mathrm{PSL}_2(\mathbb{F}_p)$ has no product-free subset of size $cn^{8/9}$, where $n$ is the order of $G$. A feature of this group that plays an essential
role in the proof is that the minimal degree of a non-trivial representation of $G$ is $O(p)$. This property of 
$G$, called quasi-randomness by Gowers, is due to Frobenius and has been generalized by Landazuri and Seitz~\cite{Landazuri} to other families of finite simple
groups of Lie type.

Apart from its intrinsic interest, this theorem has found several important applications. To name a few, Nikolov and Pyber~\cite{pyber}, used Gowers' theorem to obtain an improved version of a recent theorem of Helfgott~\cite{helf} and Shalev~\cite{shalev} on product decompositions of finite simple groups. Gowers' method has also been used in studying the image of the word maps on finite simple groups~\cite{shalev2,shalev3}.

The focus of this paper will be quasi-randomness for compact groups and, more specifically, profinite groups. We will be 
interested in the family $\G( \ZZ/(p^n \ZZ))$ where $\G$ is either the special linear or symplectic group.
Our goal is to establish a lower bound on the minimal degree of all non-trivial representations and also the minimal 
degree of the faithful representations of these groups. We will then introduce the functional analytic ingredients needed to carry over Gowers' argument from finite to compact groups. These together establish that the supremal measure of a product-free set in 
these groups has an exponential rate of decay (as a function of rank) with explicit lower and upper bounds for the exponential rate. In the same vain, we prove an analogous  result for a group of the automorphisms of the $k$-regular trees. Interestingly, in this case, the bounds are of the form $O(k^c)$ for a constant $c$. Let us first set the notations and definitions.

\begin{definition}
For a group $G$, the smallest degree among all non-trivial complex representations of $G$ will be denoted by $m(G)$. In other words,
\begin{equation}
m(G):=\min_{\rho\neq 1}d_\rho,
\end{equation}
where the minimum is taken over all non-trivial representations of $G$, and $d_\rho$ denotes for degree of the representation $\rho$. We say that $G$ is $k$-{\it quasi-random} if $m(G)\ge k$.
Similarly, we will denote 
\[ m_f(G):=\min_{\ker \rho = \{1 \} }d_\rho, \]
where the minimum is taken over the set of all faithful representations of $G$.
\end{definition} 
\begin{remark}
When $G$ is a topological group, all the representations are taken to be continuous. 
\end{remark}
Our first theorem gives a lower bound for the minimal degree of all non-trivial representations of several classes of classical groups over rings $\Z{p^n}$ for $p\geq 3$. This indeed extends previous work of Landazuri and Seitz, which corresponds to the case $n=1$ in this theorem.
\begin{theorem}\label{minimal-degree}
For every group ${\mathbf G}$ in the first column of the table below
 the function on the second column (the third column, respectively) gives a lower bound for the degree of any non-trivial representation (any faithful representation, respectively)
of the group ${\mathbf G}(\Z{p^n})$. In other words,
\[ m( \G( \Z{p^n})) \ge h(\G,p), \quad  m_f( \G( \Z{p^n})) \ge h_f(\G,p,n). \]
In particular, $\G(\Z{p^n})$ is $O(p^r)$-quasi-random, where $r$ is the rank of $\G$.
\[
\begin{array}{||c|c|c||}
\hline
 {\mathbf G} & h({\mathbf G},p) & h_f({\mathbf G},p,n)     \\
 \hline
 \GSL_2 &  \frac{1}{2}(p-1)   &  \frac{1}{2}(p^n-p^{n-1})   \\
 \hline
 \GSL_k & p^{k-1}-p^{k-2}  &  (p^n-p^{n-1})p^{(k-2)n} \\   
 \hline
 \GSp_{2k} &  \frac{1}{2}(p-1)p^{k-1} & \frac{1}{2}(p^n-p^{n-1})p^{(k-1)n} \\
 \hline
\end{array}
\]

\end{theorem}

It is worth mentioning that Bourgain and Gamburd~\cite{Bourgain1} used a theorem of Clifford to find the following lower bound for 
$m_f(\SL_2(\Z{p^n})$:
\begin{equation}\label{BGC}
m_f(\SL_2(\Z{p^n}))\geq \frac{p^{n-2}(p^2-1)}{2}.
\end{equation}
Even though our bound is slightly weaker than the one obtained in \cite{Bourgain1}, it is
asymptotically equivalent to that. Our method is short and elementary and has also the advantage 
that it can be easily adapted for other classes of Chevalley groups. 

As any continuous finite dimensional representation of a profinite group factors through a finite quotient, the theorem can be rephrased as a statement about the profinite groups:

\begin{theorem}\label{minimal degree profinite}
Let $\G$ be one of the groups listed in the table above and $\G( \ZZ_p)$ denote the 
compact group of $p$-adic points of $\G$. Then the degree of any non-trivial continuous representation 
of $\G(\ZZ_p)$ is at least $h({\G},p)$.
\end{theorem}

Let $G$ be a compact, Hausdorff and second countable topological group and $\mu$ denote the Haar measure on $G$, normalized so that $\mu(G)=1$. Note that since $G$ is compact, and hence unimodular, a left Haar measure is automatically right invariant. A measurable subset of $A$ is said to be product-free if $A^2 \cap A =\emptyset$. We define the product-free measure as follows 
\begin{definition}
Let $G$ be a compact group with normalized Haar measure $\mu$. Define the product-free measure of $G$ by
$$ \p(G)=\sup \{ \mu(A): A \of G \textrm{ is measurable }, A\cap A^2= \emptyset \}.$$
\end{definition}
 We will extend an inequity of  Babai-Nikolov-Pyber, known as ``mixing inequality", which was originally  proven for finite groups. Our method is to consider the compact convolution operator, and then use the spectral theorem for compact operators. Namely we will prove
\begin{theorem}[Mixing inequality]\label{normbound}
Let $G$ be a compact, Hausdorff, and second countable topological group such that any non-trivial complex continuous representation of $G$ has dimension at least $m(G)$. Let $f_1,f_2\in L^2(G)$ and suppose that at least one of $f_1,f_2$ 
has mean zero. Then
\begin{equation}
\|f_1\ast f_2\|_2\leq \sqrt{\frac{1}{m(G)}}\|f_1\|_2\|f_2\|_2.
\end{equation}
\end{theorem} 
This theorem has an immediate corollary: 
\begin{corollary}\label{main corollary}
Let $G$ be a compact, Hausdorff and second countable topological group such that any non-trivial complex continuous representation of $G$ has dimension at least $m(G)$.
Let $A,B\of G$ be two measurable sets then,
\begin{equation}
 \|1_{A}\ast 1_B-\mu(A)\mu(B)\|_2\le \sqrt{\frac{\mu(A)\mu(B)}{m(G)}}.
\end{equation}
\end{corollary}
For compact groups we can establish the following analogue of Gowers' theorem~\cite{Gowers}:
\begin{theorem}\label{main-int}
Suppose $G$ is a compact, Hausdorff and second countable topological group such that any non-trivial complex continuous representation of $G$ has dimension at 
least $m(G)$. If $A,B,C \subseteq G$ such that $\mu(A)\mu(B)\mu(C) > m(G)^{-1},$
then the set $AB \cap C$ has a positive measure. Moreover, if $m(G) \mu(A)\mu(B)\mu(C)\geq\frac{1}{\eta^2}$ then
\begin{equation}
\mu\{(x,y,z) \in A \times B \times C: xy=z \} \ge (1-\eta) \mu(A)\mu(B)\mu(C).
\end{equation}
\end{theorem}
By Theorem~\ref{minimal degree profinite} and Theorem~\ref{main-int} and a method discussed in Section~\ref{lower bound} we can establish  upper and lower bounds on the product-free measure of some profinite groups.  
\begin{theorem}
The product-free measure of the profinite groups $\SL_k(\ZZ_p)$ and $\Sp_{2k}(\ZZ_p)$ is given by
\begin{equation}
\begin{split}
\frac{1}{p+1} & \le \p\left(\SL_2(\ZZ_p)\right)  \le \left( \frac{p-1}{2} \right) ^{-1/3}, \\
\frac{p-1}{p^k-1} & \le \p\left(\SL_k(\ZZ_p)\right)  \le (p^k-p^{k-1})^{-1/3}, \quad\quad\quad\quad k \ge 3 \\
\frac{p-1}{p^{2k}-1}  & \le \p\left(\Sp_{2k}(\ZZ_p)\right) \le \left(\frac{1}{2}(p-1)p^{k-1}\right)^{-1/3}, \,\quad k \ge 2 \\
\end{split}.
\end{equation}
\end{theorem}

The upper bounds have the following implication:

\begin{corollary}\label{main-cor}
If $A$ is a measurable subset of the groups $G=\G(\ZZ_p)$ as defined in Theorem~\ref{minimal degree profinite} with $\mu(A) > h(\G,p)^{-1/3}$, then $A^3=G.$
\end{corollary}
\begin{proof}
For every $g\in \G(\ZZ_p)$, set $B=A$ and $C=gA^{-1}$. Since 
$
\mu(A)\mu(B)\mu(C)=\mu(A)^3> h(\G,p)^{-1}
$,
Theorem~\ref{minimal degree profinite} and Theorem~\ref{main-int} show that, $AB \cap C\neq\emptyset$. If $x\in AB \cap C$ then $x=ga_3^{-1}=a_1a_2$ for $a_1,a_2,a_3 \in A$ which proves the claim.
\end{proof}
Using a similar method, we can obtain lower and upper bounds for a particular subgroup $A^+_{k+1}$
of the automorphism group of a rooted regular tree. For the definition of this subgroup we refer the reader to Section~\ref{Automorphisms of the regular tree}.
 \begin{theorem}\label{tree}
For all $k \ge 6$ we have
\begin{equation}
\frac{1}{k+1} \le \p(A^+_{k+1})  \le \frac{1}{(k-1)^{1/3}}.
\end{equation}
\end{theorem}
Using a result of Green and Ruzsa~\cite{GR}, we can also compute the exact value of product-free measure of the additive group of $p$-adic integers. Favouring a consistent terminology, we continue to use product-free (rather than sum-free) for subsets of these additive groups.  

\begin{theorem}\label{abel}
The product-free measure of the additive groups of $p$-adic integers $\ZZ_p$ and power series $\FF_p[[t]]$ are respectively given by,
\begin{equation}\label{formula}
\begin{split}
\p(\ZZ_p) &=\pfunction{1/3+1/(3p)}{p \equiv 2 \mod 3}{1/3}{\emph{otherwise}}. \\
\p(\FF_p[[t]]) &=\ppfunction{1/3+1/(3p)}{p \equiv 2 \mod 3}{1/3}{p=3}{1/3-1/(3p)}{p \equiv 1 \mod 3}.
\end{split}
\end{equation}
\end{theorem} 

This paper is organized as follows: In Section~\ref{Prel}, we will recall the definitions and set some notations. Moreover, in those sections we will establish some elementary properties of the product-free measure. In Section~\ref{facts-profinite}, we gather some facts about the representation theory of profinite groups. In Section~\ref{root function} we will prove Theorem~\ref{minimal-degree}.  
Gowers' proof~\cite{Gowers} uses the language of quasirandom graphs. We will translate his argument to a functional analytic language that is more suitable for dealing with compact groups. This is done 
in Section~\ref{functional analysis}. In Section~\ref{lower bound}, we will establish the lower bounds. In  Section~\ref{Automorphisms of the regular tree}, we will prove Theorem ~\ref{tree}. Finally, in Section~\ref{Product-free measure}, we will prove Theorem~\ref{abel}.
%===================================================================================
\section{Preliminaries and Notations}\label{Prel}
Groups considered in this paper are all assumed to be compact, Hausdorff and second countable. In general the group operation
is denoted multiplicatively; we occasionally make an exception for abelian groups and shift to the additive notation. We use $\mu$ for the normalized bi-invariant Haar measure on the group. The corresponding Lebesgue spaces will be denoted by $L^p(G)$ and the respective norm is denoted by $\| \cdot \|_p$.
For a subset $A$ of a group $G$, we use $1_A$ to denote the characteristic function of $A$. For subsets
$A$ and $B$, the product set $AB$ is the set of all products of the form $ab$ where $a \in A$ and $b \in B$.
We also use the shorthand $A^2=AA$. The cardinality of a finite set $A$ will be denoted by $|A|$. 
The finite field with $p$ elements is denoted by $\FF_p$. We will be working with the ring of $p$-adic integers and the ring of formal power series over
$\FF_p$, denoted respectively by $\ZZ_p$ and $\FF_p[[t]]$. Each one of these groups is equipped with the profinite topology. 

Moving to product-free measure. First note that $\p(G) \le 1/2$. This follows from the fact that if $A \cap A^2 =\emptyset$ then for each $x \in A$, the sets $A$ and $xA$ are disjoint and have the same Haar measure. One can also easily see that for 
any non-trivial group $G$, $\p(G)>0$. Let $G$ be a compact group. It is known that the topology of $G$ is given by a bi-invariant metric (see Corollary A4.19 in~\cite{Hofmann}.) Let $d_G$ be such a metric and $D={\mathrm{diam}}(G)$ be the diameter of $G$. Let us 
also denote $f(r)=\mu(B(x,r))$ (note that the bi-invariance of $d_G$ implies that volume of the ball is independent of the center.) Then we have
\begin{proposition} For any non-trivial group $G$,  $\p(G) \ge f(D/3)>0.$
\end{proposition} 
\begin{proof}
Choose $y,z \in G$ such that $d_G(y,z)=D$ and let $x=z^{-1}y$. We have,
\[ d_G(x,x^2)=d_G(1,x)=d_G(z,zx)=d_G(z,y)=D.\]
An application of triangle inequality shows that if $u,v \in B(x,D/3)$ then $uv \in B(x^2,2D/3)$
and hence $uv \notin B(x, D/3)$. This shows that $B(x, D/3)$ is product-free.
\end{proof}
It is worth pointing out that one can give an alternative definition by replacing $A \cap A^2 =\emptyset$ with $\mu(A \cap A^2)=0$. However, this turns out to be equivalent:
\begin{proposition}
Suppose $G$ is an infinite compact group with Haar measure $\mu$. Define 
\begin{equation*}
\p_0(G) =\sup \{ \mu(A): A \of G \textrm{ is measurable }, \mu(A\cap A^2)=0\}.
\end{equation*}
Then $\p_0(G)=\p(G).$
\end{proposition}
\begin{proof}
It is clear that $\p(G)\le\p_0(G)$. To prove the inverse inequalities, let 
$A$ be a measurable set with $\mu(A \cap A^2)=0$. Then $B=A\setminus(A \cap A^2) \of A$ has the same measure as $A$
and $B \cap B^2 \of B \cap A^2 =\emptyset$. This shows that $\p(G)\le\p_0(G)$. 
\end{proof}
The following lemma is very useful.
\begin{lemma}\label{smallindex}
Let $H$ be a proper subgroup of a finite group $G$. Then $G$ contains a subset of density $[G:H]^{-1}$ which
is product-free. Similarly, if $G$ is a profinite group and $H$ is a proper open subgroup, then $G$ contains
an open product-free set of measure $[G:H]^{-1}$. 
\end{lemma}
\begin{proof}
Let $A=xH$ be a left coset of $H$ other than $H$. It is easy to see that $A^2 \cap A  =\emptyset$.
\end{proof}
%===================Some facts from profinite groups========================
\section{Complex representations of profinite groups}\label{facts-profinite}
In this section we will gather some facts about profinite groups that will be used later. Our final aim in this section is to show that any non-trivial complex continuous representation of $\SL_k(\ZZ_p)$ (respectively, $\Sp_{2k}(\ZZ_p)$) factors through a non-trivial representation of $\SL_k(\Z{p^n})$ (respectively, $\Sp_{2k}(\Z{p^n})$) for some $n$. In the next section we will find a lower bound for such a representation.

A topological group which is the projective limit of finite groups, each equipped with the discrete topology, is called a profinite group. Such a group is compact and totally disconnected. 
The Haar measure for a profinite group $G$ can be easily described as a ``limit'' of counting measures. More precisely, for an open set  $U \of G$ we have,
\begin{equation}\label{haar}
\mu(U)=\lim_{i} \frac{|\phi_i(U)|}{|G_i|}. 
\end{equation}
We call a family $\mathcal{I}$ of normal subgroups of an arbitrary group $G$ a filter base  if for all $K_1,K_2\in \mathcal{I}$ there is a subgroup $K_3\in \mathcal{I}$ which is contained in $K_1\cap K_2$. Now let $G$ be a topological group and $\mathcal{I}$ a filter base of closed 
normal subgroups, and for $K,L\in \mathcal{I}$ define $K\leq' L$ if and only if $L$ is a subgroup of $K$.  Thus $\mathcal{I}$ is a directed set with respect to the order $\leq'$ and the surjective homomorphisms $q_{KL}: G/L \rightarrow G/K$, defined for $K \leq' L$, make the groups $G/K$ into an inverse system. Write 
$\widehat{G}=\underleftarrow {\lim } (G/K).$
There is a continuous homomorphism $\theta:G \rightarrow\widehat{G}$
with kernel $\bigcap_{K\in \mathcal{I}} K$, whose image is dense in 
$\widehat{G}$. We have the following
\begin{proposition}[See~\cite{profinite}, proposition 1.2.2]\label{profinite1}
If $G$ is compact then $\theta$ is surjective; if $G$ is compact and $\bigcap_{K\in \mathcal{I}} K=\{id\}$, then $\theta$ 
is an isomorphism of topological groups.  
\end{proposition} 
Moreover we have:
\begin{proposition}[See~\cite{profinite}, proposition 1.2.1]\label{profinite2}
Let $(G,\varphi_n)$ be an inverse limit of an inverse system $(G_n)$ 
of compact Hausdorff topological groups and let $L$ be an open normal subgroup of $G$. Then $\ker\varphi_n \leq L$ 
for some $n$. 
\end{proposition}
For the profinite group $\SL_k(\ZZ_p)$, consider the following surjective homomorphism
\begin{equation}\label{K_n-pro}
0 \longrightarrow K_n  \longrightarrow \SL_k (\ZZ_p)\xrightarrow{{\varphi _n }}\SL_k (\Z{p^n}) \longrightarrow 0,
\end{equation}
where $\varphi_n$ is induced by the canonical surjective homomorphism $\ZZ_p\rightarrow\Z{p^n}$. Clearly the set $\mathcal{I}$ consists of $K_n$ is a filter base and $\bigcap K_n=\{I\}$, therefore by Proposition~\ref{profinite1} we have 
\[
\SL_k (\ZZ_p) = \underleftarrow {\lim } \, \SL_k(\Z{p^n}),
\quad  \Sp_{2n}(\ZZ_p)=\underleftarrow {\lim } \, \Sp_{2k}(\Z{p^n}).
\]  

The following proposition is a standard fact in the context of the Galois representation, however for the sake of completeness we will prove it. 
\begin{proposition}\label{profinite3}
Let $G$ be a profinite group, and assume $\rho:G\rightarrow \GL_m(\mathbb{C})$ is a continuous representation. Then the kernel of $\rho$ is an open subgroup, hence ${\mathrm{Im}}(\rho)$ is a finite subgroup of $\GL_m(\mathbb{C})$.  
\end{proposition}
\begin{proof}
It is well-known (see \cite{helgason}, ch.II, B.5) that there exists an open neighborhood of
identity  $U \subseteq \GL_{m}(\CC)$ which does not contain any non-trivial subgroup. 
Then $V:=\rho^{-1}(U)$ is an open subset of $G$ containing the identity. From the properties of profinite groups, we know that $V$ contains an open subgroup, say $H$. This implies that $\rho(H)=1$ and hence $H\leq \ker\rho$. Therefore $\ker\rho$ is open, thus ${\mathrm{Im}}(\rho)$ is finite.  
\end{proof}
We can now specialize this to $\SL_k(\ZZ_p)$. A similar result also holds for the symplectic group $\Sp_{2k}(\ZZ_p)$.

\begin{proposition}\label{factor}
Let $\rho:\SL_k(\ZZ_p)\rightarrow \GL_m(\CC)$ be a continuous non-trivial representation. Then $\rho$ factors through a non-trivial representation of $\SL_k(\Z{p^n})$ for some $n$.
\end{proposition}
\begin{proof}
By Proposition~\ref{profinite3}, $\ker\rho$ is an open normal subgroup, which by Proposition~\ref{profinite2} contains $K_n$ as introduced in~\eqref{K_n-pro} for some $n \ge 1$.
Therefore $\rho$ factors through to a non-trivial representation $
\bar{\rho}:\SL_k(\Z{p^n})\rightarrow \GL_m(\CC)$.
\end{proof}

\section{Root functions}\label{root function}
We will now turn to establishing the lower bound for the representations of the groups $\SL_{k}(\Z{p^})$ and
$\Sp_{2k}(\Z{p^n})$.

\begin{definition}
Let $\mathcal{S}$ be a family of matrices in $M_d(\CC)$. For a function $r: \mathcal{S} \rightarrow \CC$,
define 
\[ V(r):=\{ v \in \CC^d:  Sv=r(S)v \quad \textrm{for all} \ S\in \mathcal{S} \}. \]
A map $r: \mathcal{S} \rightarrow \CC$ will be called a root of $\mathcal{S}$ if $V(r) \neq \{ 0 \}$. Moreover $V(r)$ is called a root subspace. 
\end{definition}
The following proposition is a special case of Theorem 15 in section 9.5. of~\cite{Linear}.
\begin{proposition}\label{roots}
Let $\mathcal{S}$ be a commuting family of $d\times d$ unitary matrices. Then $\mathcal{S}$ has only a
finite number of roots. If $r_1, \dots, r_t$ are all the distinct roots of $\mathcal{S}$ then
\begin{enumerate}
\item $V(r_i)$ is orthogonal to $V(r_j)$ for $i \neq j$.
\item $\CC^d=V(r_1) \oplus \cdots \oplus V(r_t)$.
\end{enumerate}
\end{proposition}
%===================Representation===================================================
\subsection{Root functions for the special linear groups}\label{root-SL}
Consider the following subgroups of $\SL_k(\Z{p^n})$:
\[
L:=\left\{ \begin{pmatrix} I_{k-1}& \sigma \\ 0 & 1
\end{pmatrix}: \sigma \in (\Z{p^n})^{k-1} \right\},\qquad
 H:=\left\{ \begin{pmatrix} 
T & 0 \\ 0 & 1 
\end{pmatrix}:  T \in \SL_{k-1}(\Z{p^n}) \right\}.
\]
Notice that $L$ is an abelian subgroup and $H$ normalizes $L$. Indeed we have
\begin{equation}\label{normal}
\begin{pmatrix} 
T & 0 \\ 0 & 1 
\end{pmatrix}\begin{pmatrix} I_{k-1}& \sigma \\ 0 & 1
\end{pmatrix}\begin{pmatrix} 
T & 0 \\ 0 & 1 
\end{pmatrix}^{-1}=  
\begin{pmatrix} I_{k-1}& T \sigma \\ 0 & 1
\end{pmatrix},
\end{equation}
which means that the action by conjugation of $H$ on $L$ is isomorphic to the standard
action of $\SL_{k-1}( \Z{p^n})$ on $\left(\Z{p^n}\right)^{k-1}$. Now, let 
$\rho: \SL_k(\Z{p^n})\rightarrow \GL_d(\CC)$ 
be a non-trivial representation. 
Note that $\mathcal{S}:=\rho(L)$ is a commuting set of $ d\times d$ matrices. Next proposition shows that $H$ acts on the root functions and the root subspaces of $\mathcal{S}$.
\begin{proposition}\label{actionroot}
Let $r$ be one of the roots in the decomposition in Proposition~\ref{roots} and let $h \in H$. For any $s=\rho(l)\in\mathcal{S}$, define 
$r_h(s):=r(\rho(hlh^{-1})).$
Then $r_h$ is also a root for $\mathcal{S}$, and $V(r_h)=\rho(h^{-1})V(r)$. 
\end{proposition}
\begin{proof}
First note that since $H$ normalizes $L$, the map $r_h$ is well-defined. 
For $w \in V(r)$ and $l \in L$, we have
\[  \rho(l)(\rho(h^{-1})w)=\rho(h^{-1})(\rho(hlh^{-1})w)=r(\rho(hlh^{-1}))\rho(h^{-1})w=r_h(\rho(l))(\rho(h^{-1})w).\]
This shows that $r_h$ is a root for $\mathcal{S}$, and $\rho(h^{-1})V(r)\subseteq V(r_h)$. To show the equality let $v\in V(r_h)$, then for any $l\in L$ we have 
$
\rho(l)(\rho(h)v)=\rho(h)(\rho(h^{-1}lh)v)=r(\rho(l))(\rho(h)v)
$,
therefore we have $\rho(h)V(r_h)\subseteq V(r)$.
\end{proof}
Let $e_{i}= I+ E_{in}$, where $E_{in}$ is the matrix with all entries zero
except for $(i,n)$th entry, which is $1$. 
The group $\SL_k(\Z{p^n})$ is generated by elementary matrices, and all elementary matrices are conjugate to $e_1$, defined above. So we have:
\begin{lemma}\label{tri-SL}
If $\rho(e_1)=I$, then $\rho$ is a trivial representation. 
\end{lemma}
Now let $\rho$ be a faithful representation then we claim the following. 
\begin{lemma}\label{mainlemma-SL}
There exists a root $r$ for $\mathcal{S}$, such that $r(\rho(e_1))=\zeta$, where $\zeta$ is a primitive $p^n$-th root of unity. 
\end{lemma}
\begin{proof}
Let us denote the roots of $\mathcal{S}$ by $r_1,\dots, r_t$. Assume the contrary that for all $1\leq i\leq t$ we have $r_i(\rho(e_1))=\zeta_{p^n}^{m_i}$, where $p\mid m_i$. We will show that $\rho(e_1^{p^{n-1}})=I$, which is a contradiction since $\rho$ is a faithful representation and the order of $e_1$ is $p^n$.  By Proposition~\ref{roots} we have the decomposition 
$\CC^d=V(r_1) \oplus \cdots \oplus V(r_t).$
For $v\in \CC^d$, write $v=v_1+\dots+v_t$,
where $v_i\in V(r_i)$. Then, for any $m \in \ZZ$
\begin{equation*}
(\rho(e_1))^mv =\zeta_{p^n}^{m_1m}v_1+\dots+\zeta_{p^n}^{m_tm}v_t.
\end{equation*}
In particular, for $m=p^{n-1}$ we have $\rho(e_1^{p^{n-1}})v=v_1+\dots+v_t=v$, hence $\rho(e_1^{p^{n-1}})=I$. 
\end{proof}
\begin{proof}[Proof of Theorem~\ref{minimal-degree} for $m_f\left(\SL_k(\Z{p^n})\right)$ when $k\geq 3$:]
Let
$\rho: \SL_k(\Z{p^n})\rightarrow \GL_d(\CC),$
be a faithful representational. First note that $L$ as an abstract group is isomorphic to the direct sum of $k-1$ copies of the 
cyclic group $\Z{p^n}$ and is generated by $e_1, \dots , e_{k-1}$.
We will occasionally deviate from our standard notation for the group operation and use additive notation
for group operation on $L$, when this isomorphism is used. For instance, we will write $e_1+e_2$ instead
of $e_1 \cdot e_2$. 

By Lemma~\ref{mainlemma-SL} there is a root $r$ for $\mathcal{S}$ such that $r(\rho(e_1))=\zeta_{p^n}^{m_1}$, where $\gcd(m_1,p)=1$.
We also assume that for $2 \le i \le k-1$ we have $r(\rho(e_i))=\zeta_{p^n}^{m_i}$ where
$0 \le m_i \le p^{n}-1$. For $t\in (\Z{p^n})^*$ and $a_2, \dots ,a_{k-1} \in \Z{p^n}$ whose values will be assigned later, define
\[ \alpha:=\alpha(t,a_2,\cdots,a_{k-1})=
\left(
{\begin{array}{*{20}c}
   t & {a_1 } &  \cdots  & {a_{k - 1} } &\vline &  0  \\
   0 & {t^{ - 1} } & {} & {} &\vline &   \vdots   \\
   0 & 0 & {I_{k - 3}} & { } &\vline &  0  \\
    \vdots  &  \vdots  & {} & {} &\vline &  0  \\
\hline
   0 &  \cdots  & 0 & 0 &\vline &  1  \\
\end{array}}\right)
\in H \]

Using~\eqref{normal}, a simple computation shows that 
\[ \alpha e_1\alpha^{-1}=t e_1,\quad \alpha e_2\alpha^{-1}=t^{-1}e_2+a_2 e_1, \quad \alpha e_i\alpha^{-1}=e_i+a_i e_1 \quad (3 \le i \le k-1) .\]
By Proposition~\ref{actionroot}, we have $r_{t,a_2,\cdots,a_{k-1}}:=r_{\alpha}$ is a root and 
\begin{equation}\label{eq1}
\begin{split}
 r_\alpha(\rho(e_1))&=r(\rho(\alpha e_1\alpha^{-1}))=r(\rho(te_1))=\zeta_{p^n}^{tm_1}, \\
r_\alpha(\rho(e_2))&=r(\rho(\alpha e_2\alpha^{-1}))=r(\rho(t^{-1}e_2+a_2 e_1))=\zeta_{p^n}^{t^{-1}m_2+a_2m_1},\\
 r_\alpha(\rho(e_i))&=r(\rho(\alpha e_i\alpha^{-1}))=r(\rho(e_i+a_i e_1))=\zeta_{p^n}^{m_i+a_im_1} \quad (3 \le i \le k-1).
\end{split}
\end{equation}
Now, since $\gcd(m_1,p)=1$, by varying the values of 
$t,a_2, \dots , a_{k-1}$ we can get at least 
$$\varphi(p^n)p^{(k-2)n}=(p^n-p^{n-1})p^{(k-2)n},$$
different roots. This shows that the dimension of the representation space has to be at least 
$$(p^n-p^{n-1})p^{(k-2)n}.$$ 
\end{proof}
Similarly, we have
\begin{lemma}
Let $\rho$ be a non-trivial representation. Then, there exists a root $r$ for $\mathcal{S}$ such that $r(\rho(e_1))=\zeta_{p^n}^{m_1}$, where $m_1$ is non-zero in $\Z{p^n}$.
\end{lemma}
\begin{proof}
If for all roots we have $r_i(\rho(e_1))=1$, then similar to the proof of Lemma~\ref{mainlemma-SL} we can deduce that $\rho(e_1)=I$. But by Lemma~\ref{tri-SL} we saw that if $\rho(e_1)=I$ then $\rho$ is a trivial representation. That is a contradiction. 
\end{proof}
\begin{proof}[Proof of Theorem~\ref{minimal-degree} for $m\left(\SL_k(\Z{p^n})\right)$ when $k\geq 3$:]
Let $\rho$
be a non-trivial representation of $\SL_k(\Z{p^n})$ in $\GL_d(\CC)$. By Lemma~\ref{tri-SL}, $\rho(e_1)\neq I$ when $\rho$ is not a trivial representation. 
With the same notation used in the previous proof, we obtain the following identities similar to~(\ref{eq1}). 
\begin{equation}\label{eq2}
\begin{split}
 r_\alpha(\rho(e_1))&=r(\rho(\alpha e_1\alpha^{-1}))=r(\rho(te_1))=\zeta_{p^n}^{tm_1}, \\
r_\alpha(\rho(e_2))&=r(\rho(\alpha e_2\alpha^{-1}))=r(\rho(t^{-1}e_2+a_2 e_1))=\zeta_{p^n}^{t^{-1}m_2+a_2m_1},\\
 r_\alpha(\rho(e_i))&=r(\rho(\alpha e_i\alpha^{-1}))=r(\rho(e_i+a_i e_1))=\zeta_{p^n}^{m_i+a_im_1} \quad (3 \le i \le k-1).
\end{split}
\end{equation}
The only difference is that $m_1$ is here only non-zero in $\Z{p^n}$. So by varying the values of $t,a_2, \dots , a_{k-1}$ we can get at least $p^{k-1}-p^{k-2}$  different roots.
\end{proof}
For $\SL_2(\Z{p^n})$ this method does not works. Instead we present a different proof.
\begin{proof}[ Proof of Theorem~\ref{minimal-degree} for $m_f\left(\SL_2(\Z{p^n})\right)$:]
Let $\rho: \SL_2(\Z{p^n}) \rightarrow \GL_d(\mathbb{C})$,
be a faithful representation of the group $\SL_2(\Z{p^n})$. Set 
$
a:=
\begin{pmatrix}
1 & 1\\
0 & 1
\end{pmatrix}
$,
and let $A:=\rho(a)\neq I$. Since the order of $a$ is $p^n$ and $\rho$ is a faithful representation, therefore $A$ has a non-trivial eigenvalue, say $\zeta$, which is a primitive $p^n$-th root of unity, since otherwise $A^{p^{n-1}}=I,$ which is a contradiction. 
Notice that $a$ is conjugate to $a^m$, where $m$ is a square in $\left(\mathbb{Z}/(p^n\mathbb{Z})\right)^*$. Hereafter $m$ will be an arbitrary quadratic residue in $\mathbb{Z}/(p^n\mathbb{Z})$.  This implies that $A$ and $A^m$ would have the same set of eigenvalues. But $\zeta^m$ is an eigenvalue of $A^m$. The number of square elements in $\left(\mathbb{Z}/(p^n\mathbb{Z})\right)^*$ is $\varphi(p^n)/2$. Therefore $A$ has at least $\varphi(p^n)/2$ different eigenvalues. So $d\geq \frac{\varphi(p^n)}{2}.$
\end{proof}
For $m\left(\SL_2(\Z{p^n})\right)$, the same method gives the bound $(p-1)/2$.
%=====================================================================
\subsection{Root functions for the symplectic groups}\label{root-Sp}
Let $J$ denote the $2k\times 2k$ matrix 
$$J:=
\begin{pmatrix}
0 & I_k\\
-I_k & 0
\end{pmatrix}.$$
One way of defining the symplectic group is by:
$$
\Sp_{2k}(\Z{p^n}):=\left\{A\in \GL_{2k}(\Z{p^n}): AJA^T=J\right\}.
$$
We will use two types of symplectic matrices:
First, for any symmetric $k \times k$ matrix $\sigma$ and any invertible $k \times k$ matrix $ \alpha$,
set
$$U_{\sigma}:=\begin{pmatrix}
I_k & \sigma\\
0 & I_k
\end{pmatrix}, \quad D_{ \alpha}:= \begin{pmatrix}
\alpha & 0\\
0 &\tilde{\alpha}
\end{pmatrix}.$$
For every invertible $k \times k$ matrix $ \alpha$, set the $2k \times 2k$ matrix
where $\tilde{ \alpha}=(\alpha^{-1})^T$. 
It is easy to see that $U_{\sigma}$ and $D_{\alpha}$ defined as above are both symplectic.
From here and the fact that the set
$$
\left\{ U_{\sigma}: \sigma \in M_{k}(\ZZ), \sigma=\sigma^T \right\} \cup 
\left\{ J \right\},
$$
is a generating set for $\Sp_{2k}(\ZZ)$ (See~\cite{Rege} Section $\S 5$, Proposition 2)
and the surjectivity of the reduction map, $
\Sp_{2k}(\ZZ)\rightarrow \Sp_{2k}(\Z{p^n})$,
(See~\cite{Newman} Theorem VII.21), we can deduce that the same set reduced mod $p^n$ generates 
$\Sp_{2k}(\Z{p^n}).$
Since $U_{I_k} U_{-I_k}^T U_{I_k}=J$, then the set
\[ 
\left\{ U_{\sigma}, U^T_{\sigma}: \sigma \in M_{k}(\Z{p^n}) , \sigma=\sigma^T \right\}, \]
is also a generating set for $\Sp_{2k}(\Z{p^n})$.
(See also~\cite{Serre-Milnor}, Chapter III). Notice that for a symmetric matrix $\sigma\in M_k(\Z{p^n})$, we have $ JU_{\sigma}J^{-1}=U^T_{-\sigma}$. 
Consider the following subgroups of $\Sp_{2k}(\Z{p^n})$:
$$
L:=\left\{ U_{\sigma}: \sigma=\sigma^T\right\}, \quad H:= \left\{ D_{ \alpha}: \alpha \in \GL_k( \Z{p^n}) \right\}.$$
Notice that $L$ is an abelian group and $H$ acts on $L$ via conjugation. In fact we have,
\begin{equation}\label{matrix-mul}
D_{ \alpha} U_{ \alpha} D_{ \alpha}^{-1}=U_{ \alpha \sigma \alpha^{_T}}.
\end{equation}
This shows that the action of $H$ on $L$ is the standard action of the general linear group on symmetric quadratic 
forms. 
For $1\leq i,j\leq k$, $E_{ij}$ denotes the symmetric $k \times k$ matrix such that the $(i,j)$ and $(j,i)$ entires are $1$ and other entries are zero. Also set $G_{ij}:=U_{E_{ij}}$.
Now, it is easy to see that if $i_1 \neq j_1$ and $i_2 \neq j_2$ then the matrices $G_{i_1j_1}$ and $G_{i_2j_2}$ are conjugate. Similarly, the matrices $G_{ii}$ and $G_{jj}$ are conjugate for all $i,j$.  

\begin{lemma}\label{trivial-repre-Sp}
Let $\rho$ be a representation of $\Sp_{2k}(\Z{p^n})$, so that 
$
\rho(G_{11})=\rho(G_{12})=I$,
then $\rho$ is a trivial representation. 
\end{lemma}
Let $a_i\in \Z{p^n}$ and $t\in\left(\Z{p^n}\right)^*$, define  
\begin{equation}\label{alpha-Sp}
\alpha=\alpha_{t,a_1,\dots,a_{k-1}}:= \left( {\begin{array}{*{20}c}
   t & {a_1 } & {a_2 } &  \cdots  & {a_{k - 1} }  \\
   0 & 1 & 0 & {} & 0  \\
   0 & 0 & 1 & {} & 0  \\
    \vdots  & {} & {} &  \ddots  &  \vdots   \\
   0 & 0 & 0 &  \cdots  & 1  \\
 \end{array} } \right)\in \GL_k(\Z{p^n}),
\end{equation}
where the only possibly nonzero entries appear in the first row and on the diagonal.

The proof of the following lemma is quite straightforward:
\begin{lemma}\label{mat-iden} With $ \alpha$ defined as in~(\ref{alpha-Sp}), we have
\begin{equation}\label{identity-symp}
\begin{split}
D_{ \alpha}G_{11}D_{ \alpha}^{-1}=G_{11}^{t^2},\quad
D_{ \alpha}G_{1j}D_{ \alpha}^{-1}&=G_{11}^{2ta_{j-1}}G_{1j}^t \qquad\qquad\hspace{.6cm}\qquad (2\leq j\leq k)\\
D_{ \alpha}G_{22}D_{ \alpha}^{-1}=G_{11}^{{a_1}^2}G_{12}^{a_1} G_{22},\quad
D_{ \alpha}G_{2j}D_{ \alpha}^{-1}&=G_{11}^{2a_1a_{j-1}}G_{1j}^{a_1}G_{12}^{a_{j-1}}G_{2j}\qquad\quad (3\leq j\leq k).
\end{split}
\end{equation}
\end{lemma}

The proof of the following propositions are similar to Proposition~\ref{actionroot} and Lemma~\ref{mainlemma-SL}:
\begin{proposition}\label{actionroot-Sp}
Let $r$ be one of the roots in the decomposition in Proposition~\ref{roots} and let $h \in H$. For any $s=\rho(l)\in\mathcal{S}$, define 
$$r_h(s):=r(\rho(hlh^{-1})).$$
Then $r_h$ is also a root for $\mathcal{S}$, and $V(r_h)=\rho(h^{-1})V(r)$. 
\end{proposition}
 
\begin{lemma}\label{mainlemma-Sp}
When $\rho$ is a faithful representation, then there exists a root $r$ for $\mathcal{S}$, such that $r(\rho(G_{11}))=\zeta$, where $\zeta$ is a primitive $p^n$-th root of unity. 
\end{lemma}

We are ready to prove Theorem~\ref{minimal-degree} for $\Sp_{2k}(\Z{p^n})$. 
\begin{proof}[ Proof of Theorem~\ref{minimal-degree} for $m_f\left(\Sp_{2k}(\Z{p^n})\right)$:] For this specific matrix $ \alpha$ defined in~(\ref{alpha-Sp}), we use the shorthand 
$\overline{\alpha}:=D_{\alpha}$.
Let $
\rho: \Sp_{2k}(\Z{p^n})\rightarrow \GL_d(\CC),
$
be a faithful representation and set $\mathcal{S}=\rho(L)$.
Pick a root $r$ for $\mathcal{S}$ such that $r(\rho(G_{11}))=\zeta_{p^n}^{m}$, where $\gcd(m,p)=1$. Such a root exists by Lemma~\ref{mainlemma-Sp}. For this root, let $r(\rho(G_{1j}))=\zeta_{p^n}^{m_j}$, for $2\leq j\leq k$. With the same notation in Proposition~\ref{actionroot-Sp} and~(\ref{identity-symp}) we have
\begin{equation}
r_{\overline{\alpha}}(\rho(G_{11}))=\zeta_{p^n}^{t^2m},\quad
r_{\overline{\alpha}}(\rho(G_{1j}))=\zeta_{p^n}^{2a_{j-1}tm+tm_j},\quad (2\leq j\leq k).
\end{equation}
Notice that the number of different square in $\left(\Z{p^n}\right)^*$ is $\varphi(p^n)/2$. So by varying $t,a_1,\dots,a_{k-1}$, we will obtain at least $(p^n-p^{n-1})p^{(k-1)n}/2$ different roots. 
\end{proof}
\begin{proof}[ Proof of Theorem~\ref{minimal-degree} for $m\left(\Sp_{2k}(\Z{p^n})\right)$:] 
Now let $\rho: \Sp_{2k}(\Z{p^n})\rightarrow \GL_d(\CC)$,
be a non-trivial representation. 
Since $\rho$ is non-trivial, at least one of $\rho(G_{11})$ and $\rho(G_{12})$ are different from the identity. We split the proof into two cases:

\noindent
{\it Case I:} Let $\rho(G_{11})\neq I$. This implies that there exists a root $r$ for $\mathcal{S}$ such that $r(\rho(G_{11}))=\zeta_{p^n}^{m}$,
where $m\neq 0$ in $\Z{p^n}$. For this root let $r(\rho(G_{1i}))=\zeta_{p^n}^{m_i}$ for $2\leq i\leq k$. So
~(\ref{identity-symp}) implies that
\begin{equation}
r_{\overline{\alpha}}(\rho(G_{11}))=\zeta_{p^n}^{t^2m},\quad
r_{\overline{\alpha}}(\rho(G_{1j}))=\zeta_{p^n}^{2a_{j-1}tm+tm_j},\quad (2\leq j\leq k).
\end{equation} 
So by varying $t,a_1,\dots,a_{k-1}$, we will obtain at least $ \frac{1}{2} (p-1)p^{(k-1)}$ different roots. 

\noindent
{\it  Case II:} Let $\rho(G_{11})=I$. This will force $\rho(G_{12})\neq I$.
Now, pick a root $r$ for $\mathcal{S}$ such that $r(\rho(G_{12}))=\zeta_{p^n}^{m}$, where $m\neq 0$ in $\Z{p^n}$. Let for this root $r(\rho(G_{2j}))=\zeta_{p^n}^{m_i}$ for $2\leq j\leq k$. Then by~\eqref{identity-symp} we have
$
r_{\overline{\alpha}}(\rho(G_{12}))=r(\rho(G_{11}^{2ta_1}))r(\rho(G_{12}^t))=r(\rho(G_{12}^t))=\zeta_{p^n}^{tm}
$.
Also
$$
r_{\overline{\alpha}}(\rho(G_{22}))=r(\rho(G_{11}^{a_1^2}))r(\rho(G_{12}^{a_1}))r(\rho(G_{22}))=r(\rho(G_{12}^{a_1}))r(\rho(G_{22}))=
\zeta_{p^n}^{a_1m}r(\rho(G_{22}))
.$$
Moreover for $3\leq j\leq k $
$$
r_{\overline{\alpha}}(\rho(G_{2j}))=r(\rho(G_{11}^{2a_1a_{j-1}}))r(\rho(G_{12}^{a_{j-1}}))r(\rho(G_{1j}^{a_1}))r(\rho(G_{2j}))=
\zeta_{p^n}^{a_{j-1}m}r(\rho(G_{1j}^{a_1}))r(\rho(G_{2j})).$$
Varying $t,a_1,\dots,a_{k-1}$ will now results in at least $(p-1)p^{k-1}$ different roots and this finishes the proof.
\end{proof}
We remark that, by Theorem~\ref{minimal-degree} and Proposition~\ref{factor}, Theorem~\ref{minimal degree profinite} also follows.

%======Hilbert-Schmidt operators and product-free sets in compacts groups=====================================
\section{Upper bound estimates for the product-free measure}\label{functional analysis}
This section gathers the functional analytic ingredients of the proof needed to cast Gowers' proof in the category of 
compact topological groups. Complete proof of these facts can be found, for instance, in~\cite{func}. After reviewing these facts, we will give the proof of Theorem~\ref{normbound} and Corollary~\ref{main corollary}.

\begin{definition}
Let $\mathcal{H}$ be a separable Hilbert space with an orthonormal basis $\{e_n\}$
and let $T\in B(\mathcal{H})$, where $B(\mathcal{H})$ denotes the space of bounded operators on $\mathcal{H}$. If the condition
\[
\sum_{n=1}^\infty \|T(e_n)\|^2<\infty,
\]
holds then $T$ is called a Hilbert-Schmidt operator.
\end{definition}
This is independent of the choice of the orthonormal basis of $\mathcal{H}$. Moreover,
for a Hilbert-Schmidt operator $T$, the value of the sum is also independent of the choice of the orthonormal basis:
\begin{equation*}
\|T\|^2_{HS}:=\sum_{n=1}^\infty \|T(e_n)\|^2.
\end{equation*}

\begin{lemma} Let $\mathcal{H}$ be a separable Hilbert space and let $T\in B(\mathcal{H})$ then
\begin{enumerate}
\item[a)] $T$ is Hilbert-Schmidt if and only if $T^*$ is Hilbert-Schmidt.
\item[b)] If either $S$ or $T$ is Hilbert-Schmidt, then $ST$ is Hilbert-Schmidt.
\item[c)] If $T$ is Hilbert-Schmidt then it is compact.
\end{enumerate}
\end{lemma}

Another useful fact is the singular value decomposition for the compact operators on a Hilbert space:

\begin{lemma}[singular value decomposition]
Let $\mathcal{H}$ denote a separable Hilbert space and $T\in B(\mathcal{H})$ a compact operator. Then there exists two orthonormal sets $\{e_n\}$ and $\{e'_n\}$ in $\mathcal{H}$ such that 
$
T(e_i)=\lambda_i e'_i$,  $T^*(e'_i)=\lambda_i e_i$, for $i=1,2,\dots$.
where 
$
\lambda_1\geq \lambda_2\geq\cdots\geq 0,
$
and for any $x\in\mathcal{H}$
\begin{equation}\label{no}
T(x)=\sum_{i\geq 1}\lambda_i\langle x,e_i\rangle e'_i.
\end{equation}
Moreover, by~(\ref{no}), we have $\|T\|_{op}=\lambda_1$.
\end{lemma}

Now, let $G$ be a compact, second countable, Hausdorff topological group with a normalized Haar measure $\mu$. As usual, set $L^2(G)=\left\{ h:G\rightarrow\CC: \| h\|_2^2 <\infty\right\}$, 
where $\| h\|_2^2:=\int_G |h|^2 d\mu.$

Moreover, let us define $L^2_0(G)$ to be the set of all functions in $L^2(G)$ which are orthogonal to the constant function $1$. 
%\begin{equation*}
%L^2_0(G):=\left\{h\in L^2(G):\int_{G}h \, d\mu=0\right\}.
%\end{equation*}  
For $f_1,f_2\in L^2(G)$, the convolution $f_1\ast f_2\in L^2(G)$ is defined by
\begin{equation*}
(f_1\ast f_2)(x):=\int_G f_1(xy^{-1})f_2(y) \, d\mu(y).
\end{equation*}

Now take a function $f_1\in L_0^2(G)$ and consider the following kernel
$K(x,y):=f_1(xy^{-1}).$
Since $f_1 \in L^2(G)$, we have $K(x,y)\in L^2(G\times G).$
So we can define the following integral operator
\begin{equation}\label{Hilbert Kernel}
\Phi_K: L^2(G)\longrightarrow L^2(G),\quad
h \longmapsto \Phi_K(h),
\end{equation}
It is clear that $\Phi_K(h)(x)=(f_1\ast h)(x)$. One can also easily see that the adjoint operator 
$ \Phi_K^\ast$ is given by  
\begin{equation*}
\Phi_K^*(h)(y)=\int_G \overline{K(x,y)}h(x)d\mu(x).
\end{equation*}

\begin{lemma}\label{Trace}The integral operator $\Phi_K: L^2(G)\rightarrow L^2(G)$ is a Hilbert-Schmidt operator and hence is compact. The norm of $\Phi_K$ is given by
$\|\Phi_K\|_{HS}=\|K\|_{L^2(G\times G)}.$
\end{lemma}

\begin{proof}[ Proof of Theorem~\ref{normbound}:]
To prove Theorem~\ref{normbound} note that by replacing $f_2$ with $f_2-\int_G f_2\, d\mu$ and noticing that $f_1\ast c=0$ for every constant function $c$, without loss of generality, we can assume that $f_2\in L_0^2(G)$.
Consider the operator $\Phi_K: L^2(G)\rightarrow L^2(G)$, 
defined by~\eqref{Hilbert Kernel}. Let $\Phi_0$ denote the restriction of $\Phi_K$ to $L_0^2(G)$. We need to show that  
\begin{equation}
\|\Phi_0\|_{op}^2\leq \frac{1}{m(G)}\|f_1\|_2^2.
\end{equation}
Apply the singular value decomposition to obtain orthonormal 
bases $\{e_n\}$ and $\{e'_n\}$ in $L_0^2(G)$ such that 
$
\Phi_0(e_i)=\lambda_i e'_i,
$ 
where 
$
\lambda_1\geq \lambda_2\geq\cdots\geq 0.
$
Moreover $\|\Phi_0\|_{op}=\lambda_1$. Let $V_1$ be the eigenspace of the self-adjoint operator $\Phi_0^*\Phi_0$ corresponding to $\lambda_1^2$. Since $\Phi_{0}^{\ast}\Phi_{0}$ is a compact
operator, $\dim V_{1} < \infty $. We have
\begin{align*}
\|\Phi_0\|_{op}^2\dim V_1=\lambda_1^2 \dim(V_1) & \leq \sum_{i=1}^\infty \lambda_i^2 
\le \|\Phi_K\|_{HS}^2  =\|K\|^2_{L^2(G\times G)}\\
&=\int_G\int_G |f_1(xy^{-1})|^2 d\mu(y)d\mu(x)=\|f_1\|_2^2.
\end{align*} 
We show that $\dim V_1\geq m(G)$, and this would finish the proof. We will construct a linear action of $G$
on $V_1$ by defining for every $h\in V_1$ and $g \in G$
$T_gh(x):=h(xg).$
We need to verify that,
 \begin{equation}\label{action}
T_g(\Phi_K^*\Phi_K(h))=\Phi_K^*\Phi_K(T_gh). 
\end{equation}
Since $G$ is compact and hence unimodular we have,
\begin{align*}
\Phi_K(T_g h)(x)&= \int_G f_1(xy^{-1})h(yg)d\mu(y)
=\int_G f_1(x(zg^{-1})^{-1})h(z) d\mu(z)\\ 
&=\int_G f_1(xgz^{-1})h(z) d\mu(z)
=T_g(\Phi_K(h))(x).
\end{align*}
By acting $\Phi_K^*$ from the left we obtain~(\ref{action}). Since $V_1$ is a subspace of $L_0^2(G)$, it does not contain the constant function, and hence this linear action is non-trivial. This induces a non-trivial representation of $G$ in the unitary group $U(V_1)$, thus $\dim V_1 \ge m(G)$. 
\end{proof}
\begin{proof}[ Proof of Corollary~\ref{main corollary}.]
Apply the inequality to $f_1=1_A$ and $f_2=1_B-\mu(B)$.
\end{proof}
\begin{proof}[ Proof of Theorem~\ref{main-int}:]
Let $ S :=\{ y \in G: (1_A \ast 1_B)(y)=0 \}$.
Thus
\begin{align*}
\mu(S)^{1/2}\mu(A)\mu(B)&=
\pra{
\int_S \left|(1_A\ast 1_B)(y)-\mu(A)\mu(B)\right|^2 d\mu(y)
}^{1/2}\\
&\leq \pra{
\int_G \left|(1_A\ast 1_B)(y)-\mu(A)\mu(B)\right|^2 d\mu(y)
}^{1/2}\\
&=\|1_A\ast 1_B-\mu(A)\mu(B)\|_2.
\end{align*}
From Corollary~\ref{main corollary} we can deduce that 
$$
\mu(S)^{1/2}\mu(A)\mu(B)\leq \sqrt{\frac{\mu(A)\mu(B)}{m(G)}},
$$
therefore 
$
\mu(S)\leq 1/(m(G)\mu(A)\mu(B))
$.
This implies that $\mu(C\setminus S)>0 $, since otherwise we get 
$
\mu(C)\mu(A)\mu(B)\leq 1/m(G)
$,
which is a contradiction. Hence there exists a set of positive measure of $y\in C$ so that $1_A\ast 1_B(y)\neq 0$, which means that $AB\cap C$ has positive measure.

For the second statement, let us define
$\Sigma:=\{(a,b,c)\in A\times B\times C: ab=c\}.$
Notice that
\begin{equation}
\mu(\Sigma)=\inn{1_A\ast 1_B}{1_C}=\inn{1_A\ast (1_B-\mu(B))}{1_C}+\mu(A)\mu(B)\mu(C).
\end{equation}
By Cauchy-Schwartz inequality we have
\begin{align*}
\inn{1_A\ast (1_B-\mu(B))}{1_C}^2&\leq \|1_A\ast (1_B-\mu(B))\|_2^2\|1_C\|_2^2\\
&=\|1_A\ast 1_B-\mu(A)\mu(B)\|_2^2\mu(C)
\leq \frac{\mu(A)\mu(B)\mu(C)}{m(G)}.
\end{align*}
If 
\[
\frac{\mu(A)\mu(B)\mu(C)}{m(G)}\leq \eta^2 \mu(A)^2\mu(B)^2\mu(C)^2,
\]
which is fulfilled by our assumption, we can deduce that 
$$
|\inn{1_A\ast (1_B-\mu(B))}{1_C}|\leq \eta \mu(A)\mu(B)\mu(C),
$$ 
and hence,
$
\mu(\Sigma)\geq \mu(A)\mu(B)\mu(C)-\eta \mu(A)\mu(B)\mu(C)=(1-\eta) \mu(A)\mu(B)\mu(C)
$.
\end{proof}
\begin{remark}
One can also establish a lower bound for $\mu(AB)$. For $f_1=1_A$ and $f_2=1_B-\mu(B)$, one has
$
\|f_2\|_2^2=\mu(B)(1-\mu(B)). 
$ 
Thus by Theorem~\ref{normbound} we have
$$
\mu(G-AB)^{1/2}\mu(A)\mu(B)\leq \sqrt{1/m(G)}\mu(A)^{1/2}\left(\mu(B)(1-\mu(B))\right)^{1/2},
$$
therefore 
$$
1-\frac{1-\mu(B)}{m(G)\mu(A)\mu(B)}\leq \mu(AB).
$$
\end{remark}
%=========================================================================

\section{Lower bounds for the product-free measure}\label{lower bound}
We will now turn to establishing the lower bounds for the product-free measure. 
Consider the reduction map $\phi: \SL_k(\ZZ_p)\rightarrow\SL_k(\Z{p}),$
and let $Q$ be the subgroup consisting of all matrices $g \in \SL_k(\Z{p})$ such that 
$g_{1k}=\cdots =g_{k-1,k}=0.$
It is clear that $Q$ is the stabilizer of a point in the action of $\SL_k(\ZZ_p)$ on the 
projective space and hence
$[\SL_k(\Z{p}):Q]=\frac{p^{k}-1}{p-1}$.
 Applying Lemma~\ref{smallindex} establishes the lower bound. 

Now, let us take the symplectic group. It is clear that the reduction map
$\phi: \Sp_{2k}(\ZZ_p)\rightarrow\Sp_{2k}(\Z{p}),$
is surjective.
Consider the natural action of $\Sp_{2k}(\Z{p})$ on $(\Z{p})^{2k}$. We know that for 
any field $F$ and $l \ge 1$, the action of the symplectic group $\Sp_{2k}(F)$ on the set of
$l$-dimensional isotropic subspaces is transitive. Since every one-dimensional subspace is
isotropic, we obtain a transitive action on this space, which is equivalent to the  
action on the projective space. By choosing $H$ to be the stabilizer
of one of a point in the projective space, and taking the inverse image under the reduction 
map, we obtain a subgroup $Q$ of index
$[ \Sp_{2k}(\ZZ_p):Q]=\frac{p^{2k}-1}{p-1}$,
which again establishes the lower bound for the symplectic group.

\section{Tree automorphism groups}\label{Automorphisms of the regular tree}
The goal of this section is to obtain lower and upper bounds on the product-free measure of the group of automorphisms of a rooted tree. Let $T=T_{k+1}$ be a regular tree of degree $k+1$. An automorphism of $T$ is defined to be a permutation of the vertices of $T$ that 
preserves incidence. The set of all automorphisms of $T$ together with the topology of point-wise convergence forms a locally compact group $G$ that acts on the set of vertices of $T$ transitively. Fix one of the vertices of $T$ as root and denote it by $O$. Now, consider the stabilizer $A_{k+1}$ of this vertex in $G$. It is easy to see that this subgroup is a profinite and hence compact group. 
To see this, note that since every $x \in A_{k+1}$ fixes $O$, it must permute the set of 
all the $k+1$ neighbouring vertices. A simple induction shows that for every $j \ge 1$ all $(k+1)k^{j-1}$ vertices of distance $j$ from $O$ must also be permuted by $x$. These  induce a family of homomorphism 
$\sigma_j:A_{k+1} \rightarrow \Sigma_{(k+1)k^{j-1}},$
where $\Sigma_m$ denotes the symmetric group on $\{1,2, \dots , m\}$. Then the following family of finite index subgroups can be used to provide a system of fundamental open neighbourhood of the identity automorphism
$ \mathscr{C}_j=\{ x \in A_{k+1}: \sigma_j(x)= { \mathrm{id}} \}. $
And $A_{k+1}$ is the inverse limit,
$A_{k+1}=\underleftarrow {\lim }A_{k+1}/\mathscr{C}_j.$
For more details we reefer the reader to~\cite{lub}.
We will now set out to define the subgroup of $A_{k+1}$ that appears in the statement of Theorem~\ref{tree}.

\begin{definition}
An automorphism $x \in A_{k+1}$ is called positive if $\sigma_j(x)$ is an even permutation for all $j \ge 1$. We will denote the group of all positive automorphisms by $A^+_{k+1}$.
\end{definition}

First, notice that $A^+_{k+1}$ is a closed subgroup of $A_{k+1}$ and hence a profinite group. In fact, the group can also be represented by 
\begin{equation}\label{A_k}
A_{k+1}^+=\underleftarrow {\lim }A_{k+1}^+/\mathscr{C}_j^+,
\end{equation}
where 
$
\mathscr{C}_j^+:=\left\{x \in A_{k+1}^+: \sigma_j(x)=\mathrm{id}\right\}.
$
In what follows, let $\Alt_{k+1}\leq\Sigma_{k+1}$ denote the alternating group on $k+1$ symbols. We will need the following fact from the representation theory of finite groups:

\begin{theorem}[See~\cite{fulton} Exercise 5.5]\label{alt}
For $k \ge 6$, the minimum dimension of non-trivial representations of $\Alt_k$ is $k-1$.
\end{theorem}

\begin{proof}[ Proof of Theorem~\ref{tree}:]
For the lower bound, note that $\sigma_1:A^+_{k+1} \rightarrow \Alt_{k+1}$ is surjective. Let $H$ be the subgroup of $\Alt_{k+1}$ consisting of those permutations that fix $k+1$. $H$ is clearly isomorphic to $\Alt_k$. Now, 
apply Lemma~\ref{smallindex} to the subgroup $\sigma_1^{-1}(H)$ to obtain an open subgroup of index $k+1$ in $A^+_{k+1}$. This establishes the lower bound.

For the upper bound, we need to show that $A^+_{k+1}$ does not have any non-trivial continuous representation of dimension less than $k-1$.

By~(\ref{A_k}) then we should prove that $F_j:=A_{k+1}^+/\mathscr{C}_j^+$ does not have any non-trivial representation of dimension less than $k-1$. Suppose $\rho$ be such a non-trivial representation.

For $j=1$, we will get $F_1=\Alt_{k+1}$, and then by Theorem~\ref{alt}, for $k\geq 5$, all the non-trivial representation has dimension bigger than $k$. For the sake of clarity and notational simplicity, we will present the argument for $j=2$. The argument readily extends to an arbitrary $j \ge 2$. It is easy to see that 
\[ F_2 \simeq \Alt_{k+1} \ltimes \underbrace {\left( {\Sigma _k  \times \Sigma _k  \times  \cdots  \times \Sigma _k } \right)^ +  }_{k + 1},\]
where   
$$\underbrace {\left( {\Sigma _k  \times \Sigma _k  \times  \cdots  \times \Sigma _k } \right)^ +  }_{k + 1}:=\left\{(\sigma_1, \dots, \sigma_{k+1})\in \underbrace {\left( {\Sigma _k  \times \Sigma _k  \times  \cdots  \times \Sigma _k } \right)}_{k + 1}:\prod_{i=1}^{k+1} \sgn(\sigma_i)=1\right\}.$$
and $\Alt_{k+1}$ acts by permuting the factors.

If the restriction of $\rho$ to $\Alt_{k+1}$ is non-trivial then we are done by Theorem~\ref{alt}. Suppose that the restriction of $\rho$ to $\Alt_{k+1}$ is trivial. Clearly 
\[ \underbrace{\Alt_k \times \cdots\times\Alt_k}_{k+1}\unlhd \underbrace{(\Sigma_k \times \cdots \times \Sigma_k)^+}_{k+1},\]
Again, we can assume that the restriction of $\rho$ to each one of the factors is trivial, since otherwise we can
apply Theorem~\ref{alt} to obtain the bound $k-1$. So $\rho$ factors through the quotient 
$$(\Sigma_k \times \cdots \times \Sigma_k)^+/(\Alt_k \times \cdots\times \Alt_k).$$
Note that since the restriction of $\rho$ to $\Alt_{k+1}$ is trivial we have
\[ \rho(\sigma_1,\sigma_2, \dots, \sigma_k,\sigma_{k+1})=\rho(\sigma_{i_1}, \sigma_{i_2}, \dots, \sigma_{i_k},\sigma_{i_{k+1}}), \]
for any even permutation $(i_1, i_2, \dots, i_k, i_{k+1})$ of the set $\{ 1, \dots ,k,k+1 \}$.
Notice that 
\begin{equation}\label{L}
\frac{\overbrace{\left(\Sigma_k\times\Sigma_k\times\cdots\times\Sigma_k\right)^+}^{k+1}}{\underbrace{\Alt_k\times\Alt_k\times\cdots\times\Alt_k}_{k+1}}\cong \mathscr{L}:=\left\{ (v_1, \dots, v_{k+1}) \in \mathbb{F}_2^{k+1}: v_1+\cdots +v_{k+1}=0 \right\}.
\end{equation}
So, $\rho$ will be trivial if we can show:

\begin{lemma}\label{}
Let $k \ge 6$ be an integers and $\mathscr{L}$ be the group defined in~(\ref{L}). Let 
$\rho: \mathscr{L} \rightarrow \GL_d(\CC)$,
be a non-trivial representation of $\mathscr{L}$ 
such that $\rho(v_1, \dots, v_{k+1})=\rho(v_{i_1}, \dots , v_{i_{k+1}})$ for any even permutation $(i_1, \dots , i_{k+1})$ of the
set $\{ 1, \dots , k+1 \}$. Then $d \ge k-1$.
\end{lemma}

\begin{proof}
We will show that $\rho$ is faithful when $k+1$ is odd and $|\ker(\rho)|\leq 2$ when $k+1$ is even. For $0\neq v\in \mathscr{L}$, define 
$I(v):=\{ 1 \le i \le k+1: v_i =1 \}$.
Let assume that $\rho$ is not a faithful representation. If for some $0\neq v\in \ker(\rho)$, we have $|I(v)|=2$, then $\ker(\rho)$ will contain every $w$ with $|I(w)|=2$, since for any $w\in \mathscr{L}$, with $|I(w)|=2$, we have $\sigma(v)=w$, for some $\sigma\in \Alt_{k+1}$. This implies that $\ker(\rho)=\mathscr{L}$, hence $\rho$ is trivial representation. Suppose $0 \neq v \in \ker(\rho)$ is chosen such that $|I(v)|$ is minimal. Since $\rho$ is non-trivial then $|I(v)|=2j>2$. Without loss of generality assume that $v=(1,1, \dots , 1, 0, \dots ,0)$ where the first $2j$ entries are equal to $1$ and the rest are zero.

If $k+1$ is odd then we can consider the $3$-cycle $\sigma=(1,2,2j+1) \in \Alt_{k+1}$. Now it is easy to see that $\sigma \cdot v-v$ has $1$ in only two positions, hence $\sigma(v)-v\in\ker(\rho)$, with $|I(\sigma(v)-v)|=2$. This shows that $\rho$ is a faithful representation when $k+1$ is odd.  

A similar argument can be made when $k+1$ is even and $|\ker(\rho)|>2$. This show that $\rho$ is faithful when $k+1$ is odd and $|\ker(\rho)|\leq 2$ when $k+1$ is even. In either case $\rho(\mathscr{L})$ has a subgroup isomorphic to $\mathbb{F}_2^{k-1}$. The set $\rho(\mathscr{L})$ can be simultaneously
diagonalized with diagonal entries being $\pm 1$. Now it is clear that $d \ge k-1$ in both cases.
\end{proof}

For $j \ge 3$, the group $F_j$ is isomorphic to an iterated semi-direct product of alternating groups as above and a similar argument establishes the lower bound on the degree of nontrivial representation. Applying Theorem~\ref{main-int} completes the proof. 
\end{proof}
%=======================================================================================
\section{Product-free measure of abelian groups}\label{Product-free measure}
We will compute the exact value $\p(G)$ for connected abelian Lie groups $G$. Let $\TT^k$ denote
the $k$-dimensional torus. Then,
\begin{theorem}
For any $k \ge 1$ we have $\p(\TT^k)=1/3$.
\end{theorem}
\begin{proof}
The proof is similar to the proof given in~\cite{Ked1} where only open sets $A$ are considered. We will show that in fact there is no need to restrict 
to consider just the open sets. We will write this part of the proof, which is valid for any compact group, using the multiplicative notation. Suppose that $A$ is a product-free subset with $\mu(A)=1/3+\beta$ for some 
$\beta >0$. First choose a compact set $K \of A$ with $\mu(K) \ge 1/3+ \beta/2$. Clearly $K$ is product-free and since $K$ is compact $d(K,K^2)=\epsilon>0$ where we use $d$ as as shorthand for $d_{\TT^k}$. Let $U$ be the $\delta$-neighborhood of $K$, i.e., the set of points $u \in \TT^k$ such that $d(u,k)< \delta$ for some $k \in K$. We will show that for $\delta$ small enough $U$ will be product-free as well. Let $u_1,u_2,u_3 \in U$. So there exist $k_1,k_2,k_3 \in K$ such that
$d(u_i,k_i)<\delta$ for $i=1,2,3$. Using the invariance of $d$ we have
\begin{equation*}
\begin{split}
d(u_2u_3,k_2k_3) &\le d(u_2u_3,k_2u_3)+d(k_2u_3,k_2k_3) \\
                 &= d(u_2,k_2)+d(u_3,k_3)< 2 \delta.
\end{split}
\end{equation*}
From here we have $d(u_1,u_2u_3) \ge d(k_1,k_2k_3)-d(k_1,u_1)-d(k_2k_3,u_2u_3) \ge \epsilon-3 \delta$.
So if we choose $\delta=\epsilon/4$ we will have $d(u_1,u_2u_3)> \epsilon/4$ which shows that $U \cap U^2 =\emptyset$.

Now let us assume that $A$ is an open product-free subset of $\TT^k=\TT^{1} \times \cdots \times \TT^{1}$ with $\mu(A)=1/3+ \beta$. Again, by possibly exchanging $\beta$ with $\beta/2$ we can assume that $A$ is a disjoint union of 
finitely many boxes of the form: $I_1 \times I_2 \cdots \times I_k$ where $I_j$ is an interval in the $j$-th copy of $\TT^1$. 
Choose a large prime number $p$. Set $\zeta=\exp(2\pi i/p)$ and let $G_p$ be the elementary abelian $p$-group in 
$\TT^k$ consisting of all elements of order $p$. Note that $G_p$ contains $p^k$ elements. Consider a box 
$I_1\times I_2 \cdots \times I_k$ and let $h_j$ be the length of $I_j$. It is easy to see that 
\[ |G_p \cap I| \ge (ph_1-1)\cdots (ph_k-1)=p^k \mu(I)+ O(p^{k-1}). \]   
By adding up over all boxes we will get $|G_p \cap A| \ge p^k \mu(A)+O(p^{k-1})$.
Since $G_p$ is a finite $p$-group, by Green-Ruzsa theorem (see Theorem~\ref{GR}) we have
$ \p(G_p)\le 1/3+1/(3p)$. Since $A$ is product-free we must have
$(1/3+\beta/2)+O(1/p) \le 1/3+1/(3p)$, 
 which as $p \to \infty$ gives a contradiction.
\end{proof}
For finite abelian groups, the exact value of $\p(G)$ is explicitly given by:
\begin{theorem} (Green-Ruzsa, cf.~\cite{GR})\label{GR} Suppose $G$ is a finite abelian group of size $n$. 
\begin{enumerate}
\item If $n$ is divisible by a prime $p \equiv 2 \pmod{3}$, then $\p(G)=1/3 +1/(3p)$ where $p$ is the smallest 
such $p$.
\item Otherwise, if $3|n$, then $\p(G)=1/3$.
\item Otherwise, $\p(G)=1/3-1/(3m)$ where $m$ is the largest order of any element of $G$.
\end{enumerate}
\end{theorem}
Using Theorem~\ref{GR} we will prove our first theorem. 
\begin{proof}[ Proof of Theorem~\ref{abel}:]
First we will give the proof for $\ZZ_p$. Let $\phi_n: \ZZ_p \rightarrow \Z{p^n}$ be reduction modulo $p^n$ for $n \ge 1$. For
$p \equiv 2 \pmod{3}$, it is easy to verify that if $S \of \Z{p}$ is a product-free set of density $1/3+1/(3p)$, 
provided by Green-Ruzsa theorem, then $\phi_1^{-1}(S) \of \ZZ_p$ will be a set
of the same density. For $p \equiv 1 \pmod{3}$, consider the subset of $\Z{p^n}$:
\[ S_n = \left\{ \left\lfloor \frac{p^n+1}{3} \right\rfloor  , \dots,  2 \left\lfloor \frac{p^n+1}{3} \right\rfloor -1 \right\}.  \]
By Lemma~\ref{smallindex} we have
\[ \p(\ZZ_p) \ge \sup_{n \ge 1} \frac{|S_n|}{p^n}= \sup_{n \ge 1} \frac{ \left\lfloor \frac{p^n+1}{3} \right\rfloor -1 }{p^n} =\frac{1}{3}. \]
On the other hand, suppose $A$ is a measurable product-free subset of $\ZZ_p$ with $\mu(A)$ larger than the function given on the right side of~(\ref{formula}), that we denote it by $f(p)$. Choose a compact subset $A_1 \of A$ such that $\mu(A_1)=f(p)(1+\epsilon)$ for some $\epsilon >0$. By~(\ref{haar}), this can be seen in a sufficiently finite quotient of $\ZZ_p$, i.e., for sufficiently large $n$, the set $\phi_n(A_1) \of \Z{p^n}$ has a density larger that $f(p)(1+\epsilon/2)$. By the theorem of Green and Ruzsa, this implies that 
there exist $x_n,y_n,z_n \in A_1$ such that $\phi_n(x_n+y_n-z_n)=0$. Since $A_1$ is compact, after passing to a subsequence, there exist $x,y,z \in A_1$ such that $x_n \to x, y_n \to y, z_n \to z$. Now, since $x_n+y_n-z_n \to 0$, we have $x+y=z$, which is a contradiction. 

The proof for $\FF_p[[t]]$ is similar. The only difference is that all of the finite quotients of $\FF_p[[t]]$ 
are elementary $p$-groups. Hence when $p \equiv 1 \pmod{3}$, it is the third condition in Green-Ruza theorem that 
applies. 
\end{proof}
%============================================================================
\bigskip
\noindent\textit{Acknowledgments.}
We have benefited from some notes on Terence Tao's weblog as well as Emmanuel Breuillard's lecture notes on ``Th\'{e}orie des groupes approximatifs". We wish to thank them for providing these notes online. For many fruitful discussions, we wish to thank Andrew Granville. The first author was supported in part by Facult\'{e} des \'{E}tudes
Sup\'{e}rieures et Postdoctorales de l'Universit\'{e} de Montr\'{e}al. The second author would like to 
thank CRM in Montreal for the visit during which part of this joint work was done.
We would also like to thank the referee for some useful suggestions. 
%====================================================================================
\bibliographystyle{plain}

\end{document}